\documentclass[12pt]{amsart}

\usepackage[colorlinks]{hyperref}
\usepackage{fullpage}
\usepackage{stmaryrd}
\usepackage{graphicx}
\usepackage{hyperref}
\usepackage{amsmath,amssymb,amsfonts,mathrsfs,amscd}
\usepackage{newtxtext,newtxmath}
\usepackage[all,cmtip]{xy}

\newtheorem{thm}{Theorem}[section]
\newtheorem{lem}[thm]{Lemma}

\newtheorem{thmA}{Theorem}

\newtheorem{corA}[thmA]{Corollary}

\theoremstyle{definition}

\theoremstyle{remark}

\numberwithin{equation}{section}

\newcommand{\NM}{\vartriangleleft}

\DeclareMathOperator{\Irr}{Irr}
\DeclareMathOperator{\Hall}{Hall}
\DeclareMathOperator{\B}{B}

\DeclareMathOperator{\I}{I}

\DeclareMathOperator{\Lin}{Lin}

\DeclareMathOperator{\IBr}{IBr}

\begin{document}

\title{Lifts of Brauer characters in characteristic two, II}

\author{Junwei Zhang}
\address{School of Mathematical Sciences, Shanxi University, Taipan, 030006, China.}
\email{zhangjunwei@sxu.edu.cn}
\email{changxuewu@sxu.edu.cn}
\email{jinping@sxu.edu.cn}
\email{wanglei0115@sxu.edu.cn}

\author{Xuewu Chang}

\author{Ping Jin*}

\author{Lei Wang}


\thanks{*Corresponding author}


\date{}

\maketitle

\begin{abstract}
In 2007, J. P. Cossey conjectured that if $G$ is a finite $p$-solvable group and $\varphi$
is an irreducible Brauer character of $G$ with vertex $Q$,
then the number of lifts of $\varphi$ is at most $|Q:Q'|$.
In this paper we revisited Cossey's conjecture for $p=2$
from the perspective of Navarro vertices
and obtained a new way to count the number of lifts of $\varphi$.
Some applications were given.
\end{abstract}

\section{Introduction}
This paper is a continuation of \cite{JW2023}, and we will still use the notation there.
Let $p$ be a prime number. Recall that given a $p$-solvable group $G$ and $\varphi\in\IBr_p(G)$ with vertex $Q$,
Cossey conjectured in \cite{C2007} that $|L_\varphi|\le |Q:Q'|$, where $L_\varphi$ denotes the set of irreducible
complex characters of $G$ that are lifts of $\varphi$.
If $p$ is odd, Cossey and Lewis established the uniqueness of the \emph{generalized vertices}
for each lift of $\varphi$ and then confirmed Cossey's conjecture in many important situations;
see the references in \cite{JW2023}.
If $p$ is even, however, it seems that their argument does not work in general,
and in this case, the third and the fourth authors
introduced in \cite{JW2023} the so-called \emph{twisted vertices}
and \emph{$*$-vertices} for each irreducible complex character of $G$,
and using the uniqueness of such vertices
and the technique of $\pi$-induction developed by Isaacs in \cite{I1986},
they proved the following, which is Theorem B of that paper.

\begin{thm}\label{JW-B}
Let $G$ be a solvable group, and let $\varphi\in\IBr_2(G)$ with vertex $Q$.
If $\delta$ is a linear character of $Q$, then
$$|\tilde L_\varphi(Q,\delta)|\le |{\bf N}_G(Q):{\bf N}_G(Q,\delta)|$$
and in particular $|{\tilde L}_\varphi|\le |Q:Q'|$,
where $\tilde L_\varphi$ denotes the set of lifts of $\varphi$ that have a linear Navarro vertex
and $\tilde L_\varphi(Q,\delta)$ is the subset of those characters in $\tilde L_\varphi$ with twisted vertex $(Q,\delta)$.
\end{thm}

It should be pointed out that the proof of the above theorem is quite complicated and is far from being natural and concise.
Our motivation in this paper is to revisit Theorem \ref{JW-B} to simplify its statement and proof.
Using some of ideas and arguments developed by Navarro in \cite{N2002a},
we are now able to replace the twisted vertices and $*$-vertices introduced in \cite{JW2023}
with Navarro vertices and obtain a new way to count the number of lifts of a given Brauer character,
which seems more natural and straightforward.
As usual, we will work in the context of Isaacs' $\pi$-partial characters,
and we refer the reader to \cite{I2018} for definitions and results used in this paper.
Here we just mention that $\I_\pi(G)$ denotes the set of irreducible $\pi$-partial characters of a $\pi$-separable group $G$,
so that $\I_\pi(G)=\IBr_p(G)$ whenever $\pi=\{p\}'$.
Also, we will always write $L_\varphi(Q,\delta)$ for the set of those lifts of $\varphi\in\I_\pi(G)$ with Navarro vertex $(Q,\delta)$.
Note that in the case where $2\in\pi$,
all the generalized vertices for each lift of $\varphi$ are conjugate (see Theorem 4.1 of \cite{CL2012}),
so our notation $L_\varphi(Q,\delta)$ is compatible with
the same notation used by Cossey and Lewis in \cite{CL2010}.

The following is our main result (compare with Theorem 5.5 of \cite{JW2023}).

\begin{thmA}\label{A}
Let $G$ be a $\pi$-separable group with $2\notin\pi$,
and let $\varphi\in\I_\pi(G)$ have vertex $Q$. Then
$$|L_\varphi(Q,\delta)|\le |{\bf N}_G(Q):{\bf N}_G(Q,\delta)|$$
for each $\delta\in\Irr(Q)$. In particular, $|{\tilde L}_\varphi|\le |Q:Q'|$.
\end{thmA}

The proof of Theorem A is inspired by some arguments in \cite{CL2010}.
Actually, we will establish the following Theorem B without any restriction on the set $\pi$ of primes, from which Theorem A follows easily.
We need to introduce one more notation that will be frequently used later.
Following Lewis \cite{L2010}, we use $I_\varphi(R|Q)$ to denote the set of irreducible $\pi$-partial characters of $R$
that have vertex $Q$ and induce $\varphi$, where $G$ is a $\pi$-separable group, $\varphi\in\I_\pi(G)$
and $R$ is a subgroup of $G$ containing $Q$.

\begin{thmA}\label{B}
Let $G$ be a $\pi$-separable group for a set $\pi$ of primes, and suppose that $\varphi\in\I_\pi(G)$ has vertex $Q$.
Assume that $|I_\sigma(R|Q)|\le |{\bf N}_S(Q):{\bf N}_R(Q)|$
for any subgroups $R,S$ of $G$ with $Q\le R\le S$ and any $\sigma\in I_\varphi(S|Q)$.
Then for each $\delta\in\Irr(Q)$,
$$|L_\varphi(Q,\delta)|\le |{\bf N}_G(Q):{\bf N}_G(Q,\delta)|.$$
\end{thmA}

As another application of Theorem B, we consider an extreme case where the prime set $\pi$ consists of a single prime $p$.
We mention that the set $\I_p(G)$ of irreducible $\{p\}$-partial characters of $G$
plays an important role in studying $M$-groups; see Theorem F of \cite{I1996} and Theorem 1 of \cite{L2010a} for example.

\begin{thmA}\label{C}
Suppose that $\varphi\in\I_\pi(G)$ has vertex $Q$, where $G$ is a $\pi$-separable group.
If $\pi$ consists of just a single prime, then $$|L_\varphi(Q,\delta)|\le |{\bf N}_G(Q):{\bf N}_G(Q,\delta)|$$
for each $\delta\in\Irr(Q)$.
\end{thmA}

The following is the Brauer character version of Theorem C,
which is useful only when $G$ has even order
since Cossey established his conjecture for odd-order groups (see Theorem 1.2 of \cite{C2007}).

\begin{corA}
Let $G$ be a group of order $p^aq^b$, where $p, q$ are primes,
and suppose that $\varphi\in\IBr_p(G)$ has vertex $Q$.
Then for each $\delta\in\Irr(Q)$,
$$|L_\varphi(Q,\delta)|\le |{\bf N}_G(Q):{\bf N}_G(Q,\delta)|.$$
In particular, if $p>2$, then $|L_\varphi|\le |Q:Q'|$.
\end{corA}

Throughout this paper,
we will frequently use the following notation, often without explicit explanation.
Given a complex character $\chi$ of a $\pi$-separable group $G$,
we always use $\chi^0$ to denote
restriction of $\chi$ to the set of $\pi$-elements of $G$.
If $N\NM G$ and $\theta$ lies in $\Irr(N)$ or $\I_\pi(N)$, we write
$G_\theta$ for the inertial group of $\theta$ in $G$.
Also, if $\theta\in\Irr(N)$ is $\pi$-factored, then we use $\theta_\pi$ and $\theta_{\pi'}$
to denote the unique $\pi$-special factor and $\pi'$-special factor of $\theta$, respectively,
so that $\theta=\theta_\pi\theta_{\pi'}$.

\section{Navarro vertices and normal subgroups}
In this section, we briefly review some properties of Navarro vertices
from \cite{N2002a} needed for our proofs.

Let $G$ be a $\pi$-separable group, $Q$ a $\pi'$-subgroup of $G$ and $\delta\in\Irr(Q)$.
Given $\chi\in\Irr(G)$, recall that the pair $(Q,\delta)$ is a Navarro vertex for $\chi$
if there exists a normal nucleus $(W,\gamma)$ of $\chi$ such that
$\gamma$ is $\pi$-factored with $\gamma^G=\chi$,
$Q$ is a Hall $\pi'$-subgroup of $W$ and $\delta$ is
the restriction to $Q$ of the $\pi'$-special factor $\gamma_{\pi'}$ of $\gamma$.
In this situation, we will say that the Navarro vertex $(Q,\delta)$
is afforded by the normal nucleus $(W,\gamma)$ for $\chi$.
By construction, we see that the normal nucleus for $\chi$ is uniquely defined up to conjugacy,
so all the Navarro vertices for $\chi$ are conjugate in $G$.
For more details, see, for example, Section 3 of \cite{C2007} or Section 2 of \cite{JW2023} .
Following Navarro \cite{N2002}, we write $\Irr(G|Q,\delta)$ for the set of all members of $\Irr(G)$ having Navarro vertex $(Q,\delta)$.

We need an easy preliminary result; see the comments before Lemma 5.1 of \cite{N2002a}.
For the reader's convenience, we include a proof here.

\begin{lem}\label{nuc-conj}
Let $\chi\in\Irr(G|Q,\delta)$, where $G$ is a $\pi$-separable group.
Then any two normal nuclei that afford the Navarro vertex $(Q,\delta)$ for $\chi$ are ${\bf N}_G(Q,\delta)$-conjugate.
\end{lem}
\begin{proof}
Let $(W,\gamma)$ and $(U,\rho)$ be two normal nuclei for $\chi$ that afford the Navarro vertex $(Q,\delta)$.
Then by definition, we have $Q\in\Hall_{\pi'}(W)\cap \Hall_{\pi'}(U)$ and $\delta=(\gamma_{\pi'})_Q=(\rho_{\pi'})_Q$.
By the uniqueness (up to conjugacy), we see that $(U,\rho)=(W,\gamma)^x$ for some $x\in G$,
which implies that $U=W^x$ and $\rho=\gamma^x$.
In particular, we have $\rho_{\pi'}=(\gamma_{\pi'})^x$ (see Theorem 2.2 of \cite{I2018}).
Now both $Q$ and $Q^x$ are Hall $\pi'$-subgroups of $U=W^x$,
so there is an element $w\in W$ such that $Q^{wx}=Q$.
It follows that $wx\in{\bf N}_G(Q)$, and thus
$$\delta^{wx}=((\gamma_{\pi'})_Q)^{wx}=((\gamma_{\pi'})^{wx})_Q
=((\gamma_{\pi'})^x)_Q=(\rho_{\pi'})_Q=\delta.$$
This shows that $wx\in {\bf N}_G(Q,\delta)$ and thus $(U,\rho)=(W,\gamma)^x=(W,\gamma)^{wx}$.
\end{proof}

Now, let $N\NM G$, and consider $\chi\in\Irr(G|Q,\delta)$, where $G$ is a $\pi$-separable group.
Let $\theta$ be an irreducible constituent of $\chi_N$.
Following Navarro \cite{N2002a}, we say that $\theta$ is \emph{good} for $\chi$
or that $\theta$ is a \emph{good constituent} of $\chi_N$ (depending on the pair $(Q,\delta)$),
if there exists a normal nucleus $(W,\gamma)$ of $\chi$
affording the Navarro vertex $(Q,\delta)$, and such that $N\le W$, $\theta$ lies under $\gamma$
and $\theta_\pi$ is $Q$-invariant.
Note that if $\theta$ is good for $\chi$,
then it is clear that $\theta$ is necessarily $\pi$-factored because $\gamma$ is.

\begin{lem}\label{N5.3}
Let $G$ be a $\pi$-separable group, and let $\chi\in\Irr(G|Q,\delta)$.
Suppose that $N\NM G$ and every irreducible constituent of $\chi_N$ is $\pi$-factored.
Then the following hold.

{\rm (1)} $\chi_N$ has a good constituent.

{\rm (2)} All of the good constituents of $\chi_N$ form a single ${\bf N}_G(Q,\delta)$-orbit.

{\rm (3)} If $\delta$ is linear, then each good constituent of $\chi_N$ has the form $\alpha\beta$,
where $\alpha\in\Irr(N)$ is $\pi$-special that is $Q$-invariant and $\beta\in\Irr(N)$ is $\pi'$-special with $\beta_{Q\cap N}=\delta_{Q\cap N}$. In particular, $\beta$ is uniquely determined by $\delta$ and hence is ${\bf N}_G(Q,\delta)$-invariant.
\end{lem}

\begin{proof}
Parts (1) and (2) are the $\pi$-analog of Theorem 5.3 of \cite{N2002a},
and since the proof is exactly the same (with $\pi$ in place of $p$),
we omit it here.

For (3), let $\theta$ be a good constituent for $\chi_N$.
By definition, there exists a normal nucleus $(W,\gamma)$ for $\chi$ that affords the Navarro vertex $(Q,\delta)$,
and such that $N\le W$, $\theta$ lies under $\gamma$ and $\theta_\pi$ is $Q$-invariant.
Since $(\gamma_{\pi'})_Q=\delta$, it follows that $\gamma_{\pi'}$ is linear,
and thus $(\gamma_{\pi'})_N=\theta_{\pi'}$.
Write $\alpha=\theta_\pi$ and $\beta=\theta_{\pi'}$, so that $\theta=\alpha\beta$ and $\alpha$ is $Q$-invariant.
Now $\beta_{Q\cap N}=((\gamma_{\pi'})_Q)_{Q\cap N}=\delta_{Q\cap N}$,
and since $\beta$ is the unique $\pi'$-special character of $N$ lying over $\delta_{Q\cap N}$
(Lemma 2.10 of \cite{I2018}),
it follows that $\beta$ is uniquely determined by $\delta$
and is ${\bf N}_G(Q,\delta)$-invariant.
\end{proof}

The following result is crucial in the proof of Theorem B in the introduction.

\begin{lem}\label{good}
Let $G$ be $\pi$-separable, and suppose that $\chi\in \Irr(G|Q,\delta)$
is a lift of $\varphi\in\I_\pi(G)$, where $\delta$ is a linear character of $Q$.
Assume further that $N\NM G$ and $\theta\in\Irr(N)$ is good for $\chi$.
Then $\theta^0\in\I_\pi(N)$ has $\pi$-degree and the Clifford correspondent of $\varphi$ over $\theta^0$ has vertex $Q$.
\end{lem}

\begin{proof}
By definition, there is a normal nucleus $(W,\gamma)$ for $\chi$
that affords the Navarro vertex $(Q,\delta)$, such that $N\le W$, $\theta$ lies under $\gamma$ and $\theta_\pi$ is $Q$-invariant.
Since $\delta$ is linear, it follows from Lemma \ref{N5.3} that
$\theta=\alpha\beta$, where $\alpha=\theta_\pi$ is $Q$-invariant and $\beta=\theta_{\pi'}$ is linear.
Then $\theta^0=\alpha^0\in\I_\pi(N)$, which clearly has $\pi$-degree.
Also, $\theta^0$ is $Q$-invariant because $\alpha$ is.
Now $\chi=\gamma^G$ and $\chi^0=\varphi$, so $\gamma^0$ induces $\varphi$
and thus lies under $\varphi$ (see Lemma 5.8 of \cite{I2018}). Furthermore, since $\beta$ is linear, we see that $\theta^0$ lies under $\gamma^0$
and hence it also lies under $\varphi$.

Let $T$ be the inertial group of $\theta^0$ in $G$, so that $Q\le T$.
Write $\tau\in\I_\pi(T)$
for the Clifford correspondent of $\varphi$ over $\theta^0$.
We want to show that $Q$ is a vertex for $\tau$.
Observe that $T\cap W$ is the inertial group of $\theta^0$ in $W$,
and if we denote by $\eta$ the Clifford correspondent of $\gamma^0$ over $\theta^0$,
then $\gamma^0=\eta^W$.
Since $\gamma^0(1)$ is a $\pi$-number, we deduce that both $|W:T\cap W|$ and $\eta(1)$
are $\pi$-numbers, and thus $Q$ is a Hall $\pi'$-subgroup of $T\cap W$.
It is clear that $\eta^T$ lies over $\theta^0$ and under $\varphi$,
which forces $\eta^T=\tau$.
By definition, we conclude that $Q$ is a vertex for $\tau$, as desired.
\end{proof}

\section{Proofs}
We begin with a useful result.

\begin{lem}\label{linear}
Let $G$ be a $\pi$-separable group, and suppose that
$\chi\in \Irr(G|Q,\delta)$ is a lift of $\varphi\in\I_{\pi}(G)$.
Then $Q$ is a vertex for $\varphi$ if and only if $\delta$ is a linear character of $Q$.
\end{lem}

\begin{proof}
By definition, there exist a subgroup $W$ of $G$ and some $\pi$-factored character $\gamma\in\Irr(W)$ such that $\gamma^G=\chi$,
$Q\in\Hall_{\pi'}(W)$ and $\delta=(\gamma_{\pi'})_Q$.
Then $\varphi=\chi^0=(\gamma^G)^0=(\gamma^0)^G$, and thus
$$\varphi(1)_{\pi'}=\chi(1)_{\pi'}=|G|_{\pi'}\delta(1)/|Q|.$$
If $Q$ is a vertex for $\varphi$, we deduce by Corollary 5.18 of \cite{I2018}
that $\varphi(1)_{\pi'}=|G|_{\pi'}/|Q|$,
which forces $\delta(1)=1$.
Conversely, if $\delta$ is linear, then $\gamma^0\in\I_\pi(W)$
has $\pi$-degree, and by definition, we see that $Q$ is a vertex for $\varphi$.
\end{proof}

The following is essentially Lemma 3.3 of \cite{CL2010},
and we can remove the condition that $2\in\pi$ by Lemma \ref{linear}.
Recall from the introduction that we use $L_{\varphi}(Q,\delta)$
to denote the set of irreducible characters of $G$ having Navarro vertex $(Q,\delta)$ and lifting
the $\pi$-partial character $\varphi\in\I_\pi(G)$.

\begin{lem}\label{max-N}
Let $G$ be $\pi$-separable, and
let $\varphi\in\I_{\pi}(G)$ have vertex $Q$.
Suppose that $\chi,\psi\in L_{\varphi}(Q,\delta)$.
If $N$ is a normal subgroup of $G$, then the irreducible constituents of $\chi_{N}$ are $\pi$-factored if and only if the
irreducible constituents of $\psi_{N}$ are $\pi$-factored.
\end{lem}

\begin{proof}
By Lemma \ref{linear}, we see that $\delta$ is linear,
and by symmetry, we may assume that $\chi_N$ has $\pi$-factored irreducible constituents
and we want to show that $\psi_N$ also has $\pi$-factored irreducible constituents.

To do this, we can choose a normal subgroup $L$ of $G$ contained in $N$
maximal with the property that the irreducible constituents of $\psi_{L}$ are
$\pi$-factored. If $L=N$, we are done, and so we can assume that $L<N$.
Since the irreducible constituents of $\chi_{L}$ are also $\pi$-factored,
it follows from Lemma \ref{N5.3} that there exist
$\pi$-special characters $\alpha_1,\alpha_2\in\Irr(L)$
and a $\pi'$-special character $\beta\in\Irr(L)$ with $\beta_{Q\cap L}=\delta_{Q\cap L}$,
and such that $\alpha_{1}\beta$ and $\alpha_{2}\beta$ are irreducible constituents of $\chi_{L}$ and $\psi_{L}$, respectively.
Note that $\beta$ is linear, so $(\alpha_i\beta)^0=(\alpha_i)^0$,
which is an irreducible constituent of $(\chi^0)_N=\varphi_N$ for $i=1,2$.
By Clifford's theorem for $\pi$-partial characters (Corollary 5.7 of \cite{I2018}),
we see that $(\alpha_1)^0$ and $(\alpha_2)^0$ are conjugate in $G$,
and thus $\alpha_1$ and $\alpha_2$ are also $G$-conjugate
because the map $\alpha\mapsto\alpha^0$ defines an injection from
the set of $\pi$-special characters of $L$ into $\I_{\pi}(L)$ (see Theorem 3.14 of \cite{I2018}).

Now let $M/L$ be a chief factor of $G$ with $M\leq N$,
so that $M/L$ is either a $\pi$-group or a $\pi'$-group.
Since we are assuming that $\chi_{N}$ has $\pi$-factored irreducible constituents,
each irreducible constituent of $\chi_{M}$ is also $\pi$-factored,
and in particular, some member of $\Irr(M|\alpha_{1}\beta)$ must be
$\pi$-factored. It follows by Clifford's theorem that either $\alpha_1$ or $\beta$ extends to $M$
according as $M/L$ is a $\pi'$-group or a $\pi$-group.
In particular, either $\alpha_1$ or $\beta$ is $M$-invariant.
Note that $\alpha_1$ is $M$-invariant if and only if $\alpha_2$ is $M$-invariant
because both are conjugate in $G$ and $M\NM G$,
and thus either $\alpha_2$ or $\beta$ is $M$-invariant.
By Lemma 4.2 of \cite{I2018}, we conclude that every member of
$\Irr(M|\alpha_2\beta)$ is $\pi$-factored. But $\Irr(M|\alpha_2\beta)$ contains
some irreducible constituent of $\psi_M$, and so the irreducible constituents of $\psi_M$ are
also $\pi$-factored. This contradicts the maximality of $L$, thus proving the lemma.
\end{proof}

We can now prove Theorem B in the introduction.

\medskip\noindent
\emph{Proof of Theorem B.}
We proceed by induction on $|G|$.
If $L_{\varphi}(Q,\delta)$ is empty, there is nothing further to prove,
so we can assume that this set is nonempty.
Applying Lemma \ref{linear}, we deduce that $\delta$ is linear.
By Lemma \ref{max-N}, we can fix $N\NM G$
such that for each $\chi\in L_{\varphi}(Q,\delta)$,
$N$ is the unique maximal normal subgroup of $G$ with the property that every irreducible constituent of $\chi_{N}$ is $\pi$-factored.

Write $\Omega$ for the set of those irreducible constituents $\eta$ of $\varphi_{N}$
such that the Clifford correspondent $\varphi_\eta$ of $\varphi$ over $\eta$ has vertex $Q$.
By Lemma 6.33 of \cite{I2018} (or Lemma 4.2 of \cite{N2002}),
we see that $\Omega$ cannot be empty and forms a single ${\bf N}_G(Q)$-orbit.
It is clear that each member of $\Omega$ is $Q$-invariant.
Let $\Lambda$ be a set of representatives for the ${\bf N}_{G}(Q,\delta)$-orbits in $\Omega$.
Observe that each of irreducible constituents of $\varphi_N$ has $\pi$-degree by Lemma \ref{good},
so it has a unique $\pi$-special lift (see Theorem 3.14 of \cite{I2018}).
In particular, if $\eta\in\Lambda$, then we will always use $\hat\eta$ to denote the $\pi$-special lift of $\eta$.
Also, we let $\hat\delta=\beta$ be as in Lemma \ref{N5.3}(3),
which is ${\bf N}_{G}(Q,\delta)$-invariant and is independent of the choice of $\chi\in L_{\varphi}(Q,\delta)$.
Note that since $\hat\delta$ is $\pi'$-special, it follows
from Theorem 2.2 of \cite{I2018} that $\hat\eta\hat\delta\in\Irr(N)$ for every $\eta\in\Lambda$.

Following Navarro \cite{N2002}, we denote by $\Irr(G|Q,\delta,\theta)$ the set of those characters $\chi\in\Irr(G|Q,\delta)$
such that $\theta\in\Irr(N)$ is good for $\chi$, and for notational convenience, we write
$$L_{\varphi}(Q,\delta,\theta)=\Irr(G|Q,\delta,\theta)\cap L_\varphi.$$
To complete the proof, we will carry out the following steps.

\smallskip\noindent
\emph{Step 1. We have $|L_{\varphi}(Q,\delta)|=\sum_{\eta\in\Lambda}
|L_{\varphi}(Q,\delta,\hat{\eta}{\hat\delta})|$.}

It suffices to show that $L_{\varphi}(Q,\delta)=\bigcup_{\eta\in\Lambda}L_{\varphi}(Q,\delta,\hat{\eta}{\hat\delta})$
is a disjoint union. For each $\eta\in\Lambda$, it is clear that $L_{\varphi}(Q,\delta,\hat{\eta}{\hat\delta})\subseteq L_{\varphi}(Q,\delta)$.
If $\chi\in L_{\varphi}(Q,\delta)$, then by Lemma \ref{N5.3},
there exists a good constituent $\theta$ of $\chi_N$ such that $\theta=\alpha\hat\delta$,
where $\alpha\in\Irr(N)$ is the $\pi$-special factor of $\theta$.
Furthermore, $\hat\delta$ is the linear $\pi'$-special factor of $\theta$,
which implies that $\theta^0=\alpha^0$ is irreducible.
By Lemma \ref{good}, we have $\alpha^0=\theta^0\in\Omega$,
and thus there exist $\eta\in\Lambda$ and $x\in {\bf N}_G(Q,\delta)$ such that
$(\alpha^x)^0=(\alpha^0)^x=\eta$.
It follows that $\alpha^x$ is a $\pi$-special lift of $\eta$, and thus $\alpha^x=\hat\eta$ (see Theorem 3.14 of \cite{I2018}).
We have seen that $\hat\delta$ is ${\bf N}_{G}(Q,\delta)$-invariant,
so $\theta^x=\alpha^x{\hat\delta}^x=\hat\eta{\hat\delta}$.
This shows that $\hat\eta{\hat\delta}$ is also a good constituent of $\chi_{N}$
by Lemma \ref{N5.3} again,
so $\chi\in L_{\varphi}(Q,\delta,\hat{\eta}{\hat\delta})$.
We conclude that $L_{\varphi}(Q,\delta)=\bigcup_{\eta\in\Lambda}L_{\varphi}(Q,\delta,\hat{\eta}{\hat\delta})$.

Furthermore, if $\chi\in L_{\varphi}(Q,\delta,\hat{\eta}{\hat\delta})\cap L_{\varphi}(Q,\delta,\hat{\tau}{\hat\delta})$
for some $\eta,\tau\in\Lambda$,
then both $\hat{\eta}{\hat\delta}$ and $\hat{\tau}{\hat\delta}$ are good for $\chi$.
By Lemma \ref{N5.3}(2), we have $(\hat{\tau}{\hat\delta})^y=\hat{\eta}{\hat\delta}$ for some $y\in{\bf N}_{G}(Q,\delta)$.
This implies that $\hat{\tau}^y=\hat{\eta}$
(Theorem 2.2 of \cite{I2018}), and hence $\tau^y=(\hat\tau^0)^y=\hat\eta^0=\eta$.
By definition, we obtain $\tau=\eta$.
This establishes that the sets $L_{\varphi}(Q,\delta,\hat\eta{\hat\delta})$ are disjoint pairwise,
as desired.

\medskip\noindent
\emph{Step 2. We have $|{\bf N}_{G}(Q):{\bf N}_{G}(Q,\delta)|
=\sum_{\eta\in\Lambda}|{\bf N}_{G_{\eta}}(Q):{\bf N}_{G_{\eta}}(Q,\delta)|$.}

Fix a $\pi$-partial character $\nu\in\Omega$.
We have seen that $\Omega$ is a single ${\bf N}_{G}(Q)$-orbit, and thus
$$|{\bf N}_{G}(Q):{\bf N}_{G_\nu}(Q)|=|\Omega|
=\sum_{\eta\in\Lambda}|{\bf N}_G(Q,\delta):{\bf N}_G(Q,\delta)\cap G_{\eta}|.$$
Observe that each $\eta\in\Lambda$ is ${\bf N}_{G}(Q)$-conjugate to $\nu$,
so $|{\bf N}_{G_{\nu}}(Q)|=|{\bf N}_{G_{\eta}}(Q)|$, and we obtain
\[\begin{aligned}
~|{\bf N}_{G}(Q):{\bf N}_{G}(Q,\delta)|
=&\sum\limits_{\eta\in\Lambda}|{\bf N}_{G_{\nu}}(Q):{\bf N}_{G}(Q,\delta)\cap G_{\eta}|\\
=&\sum\limits_{\eta\in\Lambda}|{\bf N}_{G_{\eta}}(Q):{\bf N}_{G}(Q,\delta)\cap
G_{\eta}|\\=&\sum\limits_{\eta\in\Lambda}|{\bf N}_{G_{\eta}}(Q):{\bf N}_{G_{\eta}}(Q,\delta)|.
\end{aligned}\]

\smallskip\noindent
\emph{Step 3. For each $\eta\in\Lambda$, we have $|L_{\varphi}(Q,\delta,\hat{\eta}{\hat\delta})|\leq|{\bf N}_{G_{\eta}}(Q):{\bf N}_{G_{\eta}}(Q,\delta)|$.}

Fix $\eta\in\Lambda$, and let $T=G_{\hat{\eta}{\hat\delta}}$, the inertial group of $\hat\eta\hat\delta\in\Irr(N)$
in $G$. We have seen that both ${\hat\delta}$ and $\eta$ are $Q$-invariant,
and since $\eta$ and $\hat\eta$ are uniquely determined each other (Theorem 3.14 of \cite{I2018}),
it follows that $\hat\eta$ is also $Q$-invariant. Thus $Q\leq G_{{\hat\delta}}\cap G_{\hat\eta}=G_{\hat{\eta}{\hat\delta}}=T$.

We can assume that $L_{\varphi}(Q,\delta,\hat{\eta}{\hat\delta})$ is not empty.
For each $\chi\in\Irr(G|Q,\delta,\hat{\eta}{\hat\delta})$,
let $\psi\in\Irr(T)$ be its Clifford correspondent over $\hat\eta{\hat\delta}$.
By the choice of $N$ and the construction of normal nuclei, it is easy to see that
$\chi$ and $\psi$ share a common normal nucleus
that affords the Navarro vertex $(Q,\delta)$.
This implies that $\psi\in\Irr(T|Q,\delta,\hat{\eta}{\hat\delta})$, and the Clifford correspondence guarantees that
the map $\psi\mapsto\psi^G=\chi$ defines a bijection
$$f:\Irr(T|Q,\delta,\hat{\eta}{\hat\delta})\rightarrow\Irr(G|Q,\delta,\hat{\eta}{\hat\delta}).$$
Furthermore, if $\chi\in L_{\varphi}(Q,\delta,\hat{\eta}{\hat\delta})$,
then the corresponding character $\psi\in\Irr(T|Q,\delta,\hat{\eta}{\hat\delta})$
has the properties that $\psi^0\in\I_{\pi}(T)$ has vertex $Q$
(see Lemma \ref{linear}), lies over $\eta$ and induces $\varphi$.
For notational convenience, we write $\Delta$
for the set of those members of $\I_\pi(T)$ that have vertex $Q$, lie over $\eta$ and induce $\varphi$.
Then $\chi\mapsto\psi^0$ defines a map from $L_{\varphi}(Q,\delta,\hat{\eta}{\hat\delta})$ into $\Delta$.
On the other hand, for each $\mu\in \Delta$,
it is clear that each $\psi'\in L_{\mu}(Q,\delta,\hat{\eta}{\hat\delta})$, if exists,
satisfies
$$f(\psi')=(\psi')^G\in \Irr(G|Q,\delta,\hat{\eta}{\hat\delta}),$$
and since $(\psi')^0=\mu$ and $\mu^G=\varphi$, it follows that $(\psi')^G$ is a lift of $\varphi$
and hence $f(\psi')\in L_{\varphi}(Q,\delta,\hat{\eta}{\hat\delta})$.
By the bijectivity of the map $f$, we conclude that
$$|L_{\varphi}(Q,\delta,\hat{\eta}{\hat\delta})|
= \sum_{\mu\in \Delta}|L_{\mu}(Q,\delta,\hat{\eta}{\hat\delta})|.$$

We may assume that $T<G$. Otherwise, $\hat\eta\hat\delta$ is $G$-invariant,
and let $\chi\in L_\varphi(Q,\delta,\hat\eta\hat\delta)$.
Then $\chi$ lies over $\hat\eta\hat\delta$.
We have seen that $N$ is the unique normal subgroup of $G$ maximal with the property
that $\chi_N$ has $\pi$-factored irreducible constituents,
and thus by the construction of normal nuclei, we have $N=G$.
In this case, each $\chi\in L_{\varphi}(Q,\delta)$ is $\pi$-factored,
and by the definition of Navarro vertices, we have $Q\in\Hall_{\pi'}(G)$ and $\delta=(\chi_{\pi'})_{Q}$.
Since $\delta$ is linear, we see that $\chi_{\pi'}(1)=1$ and thus $\varphi=\chi^0=(\chi_\pi)^0$.
It follows from Theorem 3.14 of \cite{I2018} that $\chi_\pi$ is uniquely determined by $\varphi$,
and $\chi_{\pi'}$ is uniquely determined by $\delta$ (see Corollary 3.15 of \cite{I2018}).
Thus $L_{\varphi}(Q,\delta)=\{\chi\}$, and the result follows in this case.
So we may assume that $T<G$.

Now, for each $\mu\in \Delta\subseteq \I_\pi(T)$,
since $|T|<|G|$, it follows by the inductive hypothesis applied in the group $T$ that
$|L_{\mu}(Q,\delta)|\leq|{\bf N}_{T}(Q):{\bf N}_{T}(Q,\delta)|$.
In particular, we have
$$|L_{\mu}(Q,\delta,\hat{\eta}{\hat\delta})|\leq |L_{\mu}(Q,\delta)|
\leq|{\bf N}_{T}(Q):{\bf N}_{T}(Q,\delta)|.$$

Next, we consider the Clifford correspondent $\varphi_\eta\in\I_\pi(G_\eta)$ of $\varphi$ over $\eta$.
Since $T\le G_{\hat\eta}\le G_\eta$, we have $\Delta=I_{\varphi_{\eta}}(T|Q)$
by the Clifford correspondence for $\pi$-partial characters (Lemma 5.11 of \cite{I2018}).
Furthermore, by the assumption of the theorem, we obtain
$|I_{\varphi_{\eta}}(T|Q)|\leq|{\bf N}_{G_{\eta}}(Q) : {\bf N}_{T}(Q)|$
and thus $|\Delta|\leq|{\bf N}_{G_{\eta}}(Q) : {\bf N}_{T}(Q)|$.
It follows that
$$\begin{aligned}
~|L_{\varphi}(Q,\delta,\hat{\eta}{\hat\delta})|
= &\sum\limits_{\mu\in \Delta}|L_{\mu}(Q,\delta,\hat{\eta}{\hat\delta})|\\
\leq&\sum\limits_{\mu\in \Delta}|{\bf N}_{T}(Q):{\bf N}_{T}(Q,\delta)|\\
=&|\Delta|\cdot|{\bf N}_{T}(Q):{\bf N}_{T}(Q,\delta)|\\
\leq&|{\bf N}_{G_{\eta}}(Q) :
{\bf N}_{T}(Q)|\cdot|{\bf N}_{T}(Q):{\bf N}_{T}(Q,\delta)|\\
=&|{\bf N}_{G_{\eta}}(Q):{\bf N}_{T}(Q,\delta)|.\end{aligned}$$
Since $\hat\delta$ is ${\bf N}_G(Q,\delta)$-invariant, we have ${\bf N}_G(Q,\delta)\le G_{{\hat\delta}}$,
and thus
$${\bf N}_{T}(Q,\delta)=G_{\hat\eta}\cap G_{\hat\delta} \cap {\bf N}_G(Q,\delta)
=G_\eta\cap {\bf N}_G(Q,\delta)={\bf N}_{G_\eta}(Q,\delta).$$
This proves that $|L_{\varphi}(Q,\delta,\hat{\eta}{\hat\delta})|\leq|{\bf N}_{G_{\eta}}(Q):{\bf N}_{G_{\eta}}(Q,\delta)|$,
as required.

Finally, combining the above three steps, we obtain
$$|L_{\varphi}(Q,\delta)|=\sum\limits_{\eta\in\Lambda}|L_{\varphi}(Q,\delta,\hat{\eta}{\hat\delta})|
\leq\sum\limits_{\eta\in\Lambda}|{\bf N}_{G_{\eta}}(Q):{\bf N}_{G_{\eta}}(Q,\delta)|
=|{\bf N}_{G}(Q):{\bf N}_{G}(Q,\delta)|,$$
and the proof is now complete.
\qed

\medskip
As an immediate application of Theorem B, we prove Theorem A.
Recall that as in \cite{JW2023},
we use $\tilde L_\varphi$ to denote the set of lifts of a given $\varphi\in\I_\pi(G)$ that have a linear Navarro vertex.

\medskip\noindent
\emph{Proof of Theorem A.}
By Corollary 5.4 of \cite{JW2023}, we have $|I_\sigma(R|Q)|\le |{\bf N}_S(Q):{\bf N}_R(Q)|$
whenever $Q\le R\le S\le G$ and $\sigma\in\I_\pi(S)$ with vertex $Q$,
and thus the first statement follows by Theorem B.

To show that $|{\tilde L}_\varphi|\le |Q:Q'|$,
we consider the action of ${\bf N}_G(Q)$ on the set $\Lin(Q)$ of linear characters of $Q$,
and let $\{\delta_1,\ldots,\delta_n\}$ be a set of representatives of the ${\bf N}_G(Q)$-orbits in $\Lin(Q)$.
Then we have
$$|Q:Q'|=|\Lin(Q)|=\sum_{i=1}^n|{\bf N}_G(Q):{\bf N}_G(Q,\delta_i)|.$$
On the other hand, we see that each $\chi\in\tilde L_\varphi$ has a linear Navarro vertex,
and thus by Lemma \ref{linear}, there exists $\delta\in\Lin(Q)$ such that $(Q,\delta)$
is a Navarro vertex for $\chi$.
Then $\delta$ is ${\bf N}_G(Q)$-conjugate to some $\delta_i$,
so that $(Q,\delta_i)$ is also a Navarro vertex for $\chi$.
This proves that $\tilde L_\varphi=\bigcup_{i=1}^n L_\varphi(Q,\delta_i)$,
which is clearly a disjoint union. The result follows by Theorem B again.
\qed

\medskip
Next, we establish Theorem C in the introduction.

\medskip\noindent
\emph{Proof of Theorem C.}
It suffices to verify the assumption of Theorem B.
For notational simplicity, we replace $S,R$ and $\sigma$ with $G,V$ and $\varphi$,
and we prove that
$$|I_\varphi(V|Q)|\le |{\bf N}_G(Q):{\bf N}_V(Q)|.$$

If $2\not\in\pi$, then the result follows by Corollary 5.4 of \cite{JW2023} or Theorem A.
So we can assume that $\pi=\{2\}$, and in this case, we see that $G$ is solvable.
Suppose that the result is false, and let $G$ be a counterexample such that $|G| +|G :V|$ is as small as possible.
Write $N$ for the core of $V$ in $G$, and let $K/N$ be a chief factor of $G$.
Applying Theorem 2.1 of \cite{L2010}, we obtain the following properties.

(1) $V$ is a maximal subgroup of $G$ with index a power of $2$,

(2) $G=KV$ and $K\cap V=N<V$,

(3) $\varphi_N=e\alpha$, where $\alpha\in\I_\pi(N)$ and $e\ge 1$,

(4) $\alpha(1)$ is a $\pi$-number,

(5) $\alpha^K$ is a multiple of some $\beta\in\I_\pi(K)$.\\
In this situation, we see that $|K:N|=|G:V|$ is a power of $2$,
and since both $\alpha(1)$ and $|K:N|$ are $\pi$-numbers,
it follows that $\beta(1)$ is also a $\pi$-number by Clifford's theorem for $\pi$-partial characters (Corollary 5.7 of \cite{I2018}).

Now, let $M/K$ be a chief factor of $G$, and write $U=M\cap V$.
Then by a standard fact about maximal subgroups of solvable groups
(see, for example, Lemma 7.8 of \cite{I2018} or Lemma 2.1 of \cite{IW2010}),
we know that $M/K$ is a $p$-group for some prime $p\neq 2$ and ${\bf C}_{K/N}(U/N)=1$.
In particular, $M/K$ is a $\pi'$-group and $N$ is the core of $U$ in $M$.

Let $\tilde\alpha\in\B_\pi(N)$ and $\tilde\beta\in\B_\pi(K)$
be such that $\tilde\alpha^0=\alpha$ and $\tilde\beta^0=\beta$.
As we have seen, both $\alpha$ and $\beta$ have $\pi$-degree,
so both $\tilde\alpha$ and $\tilde\beta$ are $\pi$-special (Theorem 4.12 of \cite{I2018}).
Furthermore, note that each irreducible constituent of $\tilde\alpha^K$ is $\pi$-special (Theorem 2.4(a) of \cite{I2018}),
and that $(\tilde\alpha^K)^0=(\tilde\alpha^0)^K=\alpha^K$, which is a multiple of $\beta$ by (5).
We conclude that each irreducible constituent of $\tilde\alpha^K$ is a $\pi$-special lift of $\beta$,
and hence $\tilde\alpha^K$ is necessarily a multiple of $\tilde\beta$.
Now $\tilde\alpha$ is $G$-invariant because $\alpha$ is,
and we deduce that $\tilde\beta_N=e\tilde\alpha$ and thus $\tilde\alpha^K=e\tilde\beta$.
This shows that $\tilde\alpha$ is fully ramified with respect to $K/N$,
and in particular, we have $e^2=|K:N|$.

Also, since $\tilde\alpha$ is $U$-invariant and $U/N$ is a $\pi'$-group,
there exists a unique $\pi$-special character $\rho\in\Irr(U)$ lying over $\tilde\alpha$,
and in fact, $\rho$ is an extension of $\tilde\alpha$. (See Lemma 2.4(b) of \cite{I2018}.)
Similarly, there exists a unique $\pi$-special character $\xi\in\Irr(M)$ lying over $\tilde\beta$,
and thus $\xi_K=\tilde\beta$.

We fix a $\pi$-partial character $\eta\in I_\varphi(V|Q)$ so that $\eta^G=\varphi$,
and let $\psi$ be the unique lift of $\eta$ in $\B_\pi(V)$.
Since $\alpha$ (and hence $\tilde\alpha$) is $V$-invariant,
it follows that $\eta$ lies over $\alpha$, and so $\psi_N$ is a multiple of $\tilde\alpha$.
This implies that each irreducible constituent $\rho'$ of $\psi_U$ lies over $\tilde\alpha$.
But such a character $\rho'$ both lies in $\B_\pi(U)$ (Corollary 4.21 of \cite{I2018})
and is $\pi$-factored (Lemma 4.2 of \cite{I2018}).
We conclude that $\rho'$ must be $\pi$-special (Theorem 4.12),
which forces $\rho'=\rho$, and thus $\psi_U$ is a multiple of $\rho$.

Write $\chi=\psi^G$. Then $\chi^0=(\psi^0)^G=\eta^G=\varphi$.
Furthermore, since $2\notin\pi'$, it follows by Theorem 5.2 of \cite{I2018} that $\chi\in\B_\pi(G)$.
Reasoning as we did with $\psi$, we can deduce that $\chi_M$ is a multiple of $\xi$.

Now $(\psi_U)^M=(\psi^G)_M=\chi_M$, so $\rho^M$ is a multiple of $\xi$.
Furthermore, we have
$$\rho^M(1)=|M:U|\rho(1)=|K:N|\tilde\alpha(1)=\tilde\alpha^K(1)=e\tilde\beta(1)=e\xi(1),$$
which forces $\rho^M=e\xi$ with $e^2=|M:U|$.
Since $\xi(1)\ge e\rho(1)$ by Frobenius reciprocity, we obtain
$$|M:U|\rho(1)=\rho^M(1)=e\xi(1)\ge e^2\rho(1)=|M:U|\rho(1),$$
and thus $\xi(1)=e\rho(1)$. It follows that $\xi_U=e\rho$, and in particular $[\xi_U,\xi_U]=|M:U|$.
This shows that $\xi$ vanishes on $M-U$, and since $\xi$ is a class function and $N$ is the core of $U$ in $M$,
we see that $\xi$ vanishes on $M-N$. Thus $\rho$ vanishes on $U-N$,
which implies that $\rho$ is necessarily fully ramified with respect to $U/N$.
But we have seen that $\rho$ is an extension of $\tilde\alpha$,
and we obtain a contradiction. This completes the proof.
\qed

\medskip

Finally, we prove Corollary D in the introduction.

\medskip\noindent
\emph{Proof of Corollary D.}
Note that $G$ is solvable by Burnside's theorem,
and that $\IBr_p(G)$ coincides with $\I_\pi(G)$, where $\pi=p'=\{q\}$.
Now the result follows by Theorem \ref{C}.
\qed

\section*{Acknowledgements}
This work was supported by the NSF of China (12171289) and the NSF of Shanxi Province
(20210302124077).


\end{document}